\documentclass[11pt,a4]{amsart}

\usepackage[utf8]{inputenx}

{\usepackage{lmodern}
\usepackage{textcomp}
\usepackage[english]{babel}
\usepackage[T1]{fontenc}

\usepackage{amssymb,amsthm,amsmath}
\usepackage{xcolor,color}
\usepackage{graphicx}
\usepackage[font=small,labelfont=bf,labelsep=period,tableposition=top]{caption}
\usepackage{hyperref}

\newtheorem{theorem}{Theorem}[section]
\newtheorem{lemma}[theorem]{Lemma}
\newtheorem{proposition}[theorem]{Proposition}
\newtheorem{corollary}[theorem]{Corollary}

\hypersetup
{   bookmarks=false,
	colorlinks=true,
	urlcolor=blue,
	pdfauthor = {Andrea Aspri},
	pdftitle = {Harmonic functions with cavity},
	pdfcreator = {LaTeX, TeXmaker, pdfLaTeX, Hyperref}
}

\title[Asymptotic Expansion in the half-space with cavity]
	{Asymptotic Expansion for Harmonic Functions\\ in the Half-Space with a Pressurized Cavity}
\author{A. Aspri, E. Beretta, C. Mascia }

\hoffset.5cm
\voffset-.15cm
\textheight20cm
\textwidth15cm
\oddsidemargin.35cm
\evensidemargin.35cm

\begin{document}
\baselineskip=16pt


\begin{center}\sf
{\Large Asymptotic Expansion for Harmonic Functions}\vskip.2cm
{\Large in the Half-Space with a Pressurized Cavity}\vskip.25cm
{\tt \today}\vskip.2cm
Andrea ASPRI\footnote{Dipartimento di Matematica ``G. Castelnuovo'',
	Sapienza -- Universit\`a di Roma, P.le Aldo Moro, 2 - 00185 Roma (ITALY), \texttt{\tiny aspri@mat.uniroma1.it}},
Elena BERETTA\footnote{Dipartimento di Matematica,
	Politecnico di Milano, Via Edoardo Bonardi - 20133 Milano (ITALY), \texttt{\tiny elena.beretta@polimi.it}},
Corrado MASCIA\footnote{Dipartimento di Matematica ``G. Castelnuovo'',
	Sapienza -- Universit\`a di Roma, P.le Aldo Moro, 2 - 00185 Roma (ITALY), \texttt{\tiny mascia@mat.uniroma1.it}
  {\sc and} {Istituto per le Applicazioni del Calcolo, Consiglio Nazionale delle Ricerche
  	(associated in the framework of the program ``Intracellular Signalling'')}}
\end{center}
\vskip.5cm

\begin{quote}\footnotesize\baselineskip 12pt 
{\sf Abstract.} 
In this paper, we address a simplified version of a problem arising from volcanology. 
Specifically, as reduced form of the boundary value problem for the Lam\'e system,
we consider a Neumann problem for harmonic functions in the half-space with a cavity $C$.
Zero normal derivative is assumed at the boundary of the half-space; differently, at $\partial C$,
the normal derivative of the function is required to be given by an external datum $g$, corresponding
to a pressure term exerted on the medium at $\partial C$.
Under the assumption that the (pressurized) cavity is small with respect to the distance from the boundary of the
half-space, we establish an asymptotic formula for the solution of the problem.
Main ingredients are integral equation formulations of the harmonic solution of the Neumann problem and
a spectral analysis of the integral operators involved in the problem.
In the special case of a datum $g$ which describes a constant pressure at $\partial C$, we recover a
simplified representation based on a polarization tensor.
\vskip.15cm

\vskip.15cm

\end{quote}

\section{Introduction}
The aim of the paper is to provide a detailed mathematical study of a simplified version of a problem arising
from volcanology.
The analysis can be considered as a blueprint, useful to address the original problem in a forthcoming research;
at the same time, the result has an interest on its own, entering in the stream of the asymptotic analysis for the conductivity equation, see \cite{Ammari1,Ammari-Griesmaier,Ammari-Kang,Ammari-Kang1,Friedman-Vogelius} and references therein, with the principal novelties of dealing with a case of an unbounded domain with unbounded boundary
and of a different choice of boundary datum (homogeneous on the boundary of the half-space and
non-homogeneous on the boundary of the cavity).

The geological problem is the detection of geometrical and physical features of magmatic reservoirs from changes within calderas.
The starting evidence is that the magma exerts a force aside the surrounding crust when migrates toward the earth's surface, producing appreciable horizontal and vertical ground displacements, which can be detected by a variety of modern
techniques (see the detailed description provided in \cite{Segall}).
As stated in \cite{Battaglia4D}, \textit{the main questions that emerge when monitoring volcanoes are how to constrain
the source of unrest}, that is to estimate the parameters related to its depth, dimension, volume and pressure,
\textit{and how to better asses hazards associated with the unrest}. 
For this purpose, the measurements of crust deformations are a useful tool to study magmatic processes since they
are sensitive to changes in the source pressure and volume. 

From a modelling point of view, the displacements are described using the theory of linear elasticity, by replacing
caldera with a half-space having a stress-free flat boundary, and the magma chamber by a cavity subjected to an
internal pressure; in more detail, let $v$ be the displacement vector
$v(x)=\left(v_1(x),v_2(x),v_3(x)\right)$ and
$\widehat{\nabla}v=\bigl(\nabla v+\left(\nabla v\right)^T\bigr)/2$ the strain tensor,
then the elastic model is defined by the linear system of equations
\begin{equation}\label{elastic model}
	\textrm{div}(\mathbb{C}\widehat{\nabla}v)=0\quad \textrm{in}\ \mathbb{R}^3_-\setminus C,
\end{equation}
with boundary conditions
\begin{equation}\label{elastic model boundary}
	\begin{aligned}
	&(\mathbb{C}\widehat{\nabla}v)n \cdot n=g\quad \textrm{on}\ \partial C
		&\qquad &(\mathbb{C}\widehat{\nabla}v)n \cdot \tau=0\quad \textrm{on}\ \partial C\\
	&(\mathbb{C}\widehat{\nabla}v)n=0\quad \text{on}\ \mathbb{R}^2	
		&\qquad & v(x)\rightarrow 0\quad\textrm{as}\ |x|\rightarrow +\infty,
	\end{aligned}
\end{equation}
where $C$ is the pressurized cavity, $g$ is a vector-valued function represents the pressure, $n$ the unit outer normal vector and $\tau$ the tangential one;
$\mathbb{C}$ is the isotropic elasticity tensor with Lam\'e parameters $\lambda$ and $\mu$ defined as 
\begin{equation*}
	\mathbb{C}:=\lambda I_3\otimes I_3+2\mu\mathbb{I}_{\textrm{sym}},
\end{equation*} 
with $I_3$ the identity matrix of order $3$ and $\mathbb{I}_{\textrm{sym}}$ the fourth order tensor such
that $\mathbb{I}_{\textrm{sym}} A=\widehat{A}$.
The goal is to determine $g$ and geometrical features of $C$ (position, volume...) from surface measurements
of the displacement field $v$. 

In applications, to handle more easily the inverse problem related to \eqref{elastic model}--\eqref{elastic model boundary}, the function $g$ is taken constant and equal to a vector $p\in\mathbb{R}^3$. Additionally, from a geological point of view, sometimes it is  reasonable to consider the magma chamber small compared to the distance from the boundary of the half-space, see \cite{Battaglia4D,McTigue,Segall}; by adding these hypotheses and fixing the geometry of the cavity, some efforts have been done during the last decades to find some explicit or approximate solutions to the mathematical model. The simplest approximate solution obtained by asymptotic expansions is due to McTigue when the cavity is a sphere, see \cite{McTigue}. The other few solutions available concern ellipsoidal shapes, dike and faults, see \cite{Battaglia4D,Segall}.   

With the ultimate goal to study in detail the elastic problem \eqref{elastic model}--\eqref{elastic model boundary}, establishing an asymptotic expansion in the presence of a  pressurized cavity of generic shape, in this paper we analyse a simplified scalar version of this model so as to shed light and mark the path to treat the elastic case.
Denoting by $\mathbb{R}^d_-$ the half-space and $\mathbb{R}^{d-1}$ its boundary,
we consider the Laplace equation 
\begin{equation*}
	\Delta u=0\quad \text{in}\ \mathbb{R}^d_-\setminus C_{\varepsilon}
\end{equation*}
with boundary conditions
\begin{equation*}
	\frac{\partial u}{\partial n}=g\quad \text{on}\ \partial C_{\varepsilon},\qquad
	\frac{\partial u}{\partial x_d}=0\quad \text{on}\ \mathbb{R}^{d-1},\qquad
	u(x)\rightarrow 0\quad \text{as}\ |x|\rightarrow +\infty
\end{equation*}
where $C_{\varepsilon}$ is the analogous of the pressurized cavity in the elastic case, with $\varepsilon$ a small parameter controlling its size, $g$ is a function defined on $\partial C_{\varepsilon}$ and $d\geq 3$.
Obviously, the choice to focus the attention on dimensions greater than two comes from the application we have in mind.   

\subsection*{Presentation of the main result}
The main goals of this paper are first to analyse the well-posedness of the scalar problem, find a representation formula of the solution and then determine an asymptotic expansion with respect to the parameter $\varepsilon$ of the solution. For the first two steps we do not need to assume the cavity  to be small.

It is worth noticing that in terms of well-posedness the case of the half-space and, in general, of unbounded domain with unbounded boundary, is more difficult to treat compared to bounded or exterior domains since both the control of the solution decay and integrability on the boundary are needed. Indeed, it is typical to treat these problems by means of weighted Sobolev spaces, see for example \cite{Amrouche-Bonzom}. 
In our case, we choose to use the particular symmetry of the half-space to prove the well-posedness in order to mantain a simple mathematical interpretation of the results. 
Therefore, we bring the problem back to an exterior domain in the whole space for which there is a vast literature on the well-posedness, see \cite{Folland}.

To trackle the issue of the asymptotic expansion, we consider the approach developed by Ammari and Kang based on single and double layer potentials for harmonic functions, see for example \cite{Ammari-Kang,Ammari-Kang1}. This is the reason why we search an integral representation formula of the solution. To do this, we take advantage of the explicit expression of the Neumann function for the half-space
\begin{equation*}
	N(x,y)=\varGamma(x-y)+\varGamma(\widetilde{x}-y),
\end{equation*}
where $\varGamma$ is the fundamental solution of the Laplacian and $\widetilde{x}$ is the symmetric point of $x$ with respect to the $x_d$-plane, in order to get a representation formula containing only integral contributions on the boundary of $C$. In detail, we find that
\begin{equation*}
	u(x)=\int_{\partial C}\left[N(x,y)g(y)-\frac{\partial}{\partial n_y}N(x,y)f(y)\right]d\sigma(y),
			\quad x\in\mathbb{R}_-^d\setminus C,
\end{equation*}
where $f$ is the trace of the solution $u$ on $\partial C$.
From the point of view of the inverse problem we are interested in evaluating the solution $u$ on the boundary of the half-space; since $\varGamma(x-y)=\varGamma(\widetilde{x}-y)$, for $x\in\mathbb{R}^{d-1}$, the integral formula becomes
\begin{equation*}
	\frac{1}{2}u(x)=\int_{\partial C}\left[\varGamma(x-y)g(y)-\frac{\partial}{\partial n_y}\varGamma(x-y)f(y)\right]d\sigma(y),
			\quad x\in\mathbb{R}^{d-1}.
\end{equation*}
Taking $B$ a bounded Lipschitz domain containing the origin and $z\in\mathbb{R}^d_-$ we consider $C_{\varepsilon}:=C=z+\varepsilon B$ with the assumption that $\textrm{dist}(z,\mathbb{R}^{d-1})\geq \delta_0>0$.
Therefore, defining $\widehat{g}(\zeta;\varepsilon)=g(z+\varepsilon \zeta)$, with $\zeta\in B$, and $S_B\widehat{g}$ the single layer potential, the main result holds

\begin{theorem} 
For any $\varepsilon>0$, let $g\in L^2(\partial C_\varepsilon)$ such that $\widehat{g}$ is independent on $\varepsilon$.
Then, at any $x\in \mathbb{R}^{d-1}$, we have
\begin{equation*}
	\begin{aligned}
	u_{\varepsilon}(x)
	&=2\varepsilon^{d-1}\varGamma(x-z)\int_{\partial B}\widehat{g}(\zeta)d\sigma(\zeta)\\
	&+2\varepsilon^d\nabla\varGamma(x-z)\cdot\int_{\partial B}\left\{n_{\zeta}\left(\tfrac{1}{2}I+K_B\right)^{-1}S_B\widehat{g}(\zeta)
		-\zeta\widehat{g}(\zeta)\right\}d\sigma(\zeta)+O(\varepsilon^{d+1})
	\end{aligned}
\end{equation*}
where $O(\varepsilon^{d+1})$ denotes a quantity uniformly bounded by $C\varepsilon^{d+1}$ with $C=C(\delta_0)$ which tends to infinity when $\delta_0$ goes to zero. 
\end{theorem}

Finally, with the asymptotic expansion in hand, we consider the Neumann boundary datum $g=-p\cdot n$ where $p$ is a constant vector in $\mathbb{R}^d$. This particular choice has a double purpose: to reconnect this problem with the constant boundary conditions of the elastic model and to make more explicit the integrals in the asymptotic formula.
The result we get contains the same polarization tensor obtained by Friedman and Vogelius in \cite{Friedman-Vogelius} for cavities in a bounded domain. 

The organization of the paper is the following. Section \ref{section direct problem} is divided into three parts: in the first one we recall some well-known results about harmonic functions and layer potentials; in the second one we examine the well-posedness of the scalar problem; in the third one we get the representation formula of the solution. In Section \ref{section spectral analysis}, we state and prove a spectral result crucial for the derivation of our main asymptotic expansion. In Section \ref{section asymptotic expansion} we present  and prove the theorem on the asymptotic expansion and finally we illustrate the result for the particular choice $g=-p\cdot n$.

{\small
\subsection*{Notation}
All the analysis is performed in $\mathbb{R}^d$, with $d\geq 3$; $B_r(x)\subset \mathbb{R}^d$ denotes the
ball with centre $x$ and radius $r>0$, and $\omega_d$ the area of the $(d-1)$-dimensional unit sphere. 
The scalar product between two vectors is represented by $x\cdot y$ and $n_x$ indicates the unit
outward normal vector in $x$ on the boundary of some specified domain.
The fundamental solution $\Gamma$ of the Laplace operator in $\mathbb{R}^d$, with $d\geq 3$,
is given by $\Gamma(x)=\kappa_d|x|^{2-d}$ with $\kappa_d:=1/\omega_d(2-d)$.
}

{\small
Points $x\in\mathbb{R}^d$, $d \geq 3$, are decomposed as $x=(x',x_d)$ where $x'=(x_1,\cdots,x_{d-1})$.
We denote by $\mathbb{R}^d_-$ the half-space $\{x\in\mathbb{R}^d:\ x_d<0\}$ and by $\mathbb{R}^{d-1}$ its boundary.
Given a point $x\in\mathbb{R}^d_-$, its reflected point $(x',-x_d)$ with respect to the plane $x_d=0$ is represented by $\widetilde{x}$. 
}

\section{The direct problem}\label{section direct problem}

In this Section, we analyse the boundary value problem 
\begin{equation}\label{direct problem}
	\begin{cases}\vspace{0.2cm}
	\Delta u=0 & \text{in}\ \mathbb{R}^d_-\setminus C\\ \vspace{0.2cm}
	\displaystyle\frac{\partial u}{\partial n}=g& \text{on}\ \partial C\\  \vspace{0.2cm}
	\displaystyle\frac{\partial u}{\partial x_d}=0 & \text{on}\ \mathbb{R}^{d-1}\\
	u\rightarrow 0 & \text{as}\ |x|\rightarrow +\infty
	\end{cases}
\end{equation}
where $C$ is the cavity, with a twofold aim: to establish well-posedness of the problem and to provide a representation formula.

\subsection*{Preliminaries}
The specific symmetry of the half-space permits to show well-posedness by extending the problem
to an exterior domain, viz. the complementary set of a bounded set.
Hence, it is useful to recall the classical results on the asymptotic behaviour of harmonic functions in
exterior domains.
Given a bounded domain $\varOmega\subset \mathbb{R}^d$ with $d\geq 3$, if $v$ is harmonic
in $\mathbb{R}^d\setminus \varOmega$ then $v$ is harmonic at infinity if and only if 
$v$ tends to $0$ as $|x|\rightarrow \infty$.
Moreover, there exist $r_0, M>0$, such that if $|x|\geq r_0$, the following estimates hold
\begin{equation}\label{asympotic behaviour}
	\left| v(x)\right|\leq M|x|^{2-d},\qquad
	\left|v_{x_i}(x)\right|\leq M|x|^{1-d}\qquad
	\left|v_{x_jx_k}(x)\right|\leq M|x|^{-d}
\end{equation}
for $i,j,k=1,\dots,d$.
The proof can be found in \cite{Folland}
(see also \cite{Evans, Kress}). 

The representation formula makes use of layer potentials
whose definitions we now recall; see \cite{Ammari-Kang, Folland, Kress}.
Given a bounded Lipschitz domain $\varOmega\subset\mathbb{R}^d$ 
and a function $\varphi\in L^2(\partial \varOmega)$, we introduce the integral operators
\begin{equation}\label{single and double layer definition}
	\begin{aligned}
	S_{\varOmega}\varphi(x)&:=\int_{\partial\varOmega}\varGamma(x-y)\varphi(y)d\sigma(y),
		\qquad & x\in\mathbb{R}^d,\\
	D_{\varOmega}\varphi(x)&:=\int_{\partial\varOmega}\frac{\partial}{\partial n_y}\varGamma(x-y)\varphi(y)d\sigma(y),
		\qquad & x\in	\mathbb{R}^d\setminus \partial\varOmega,\\
	\end{aligned}
\end{equation}
which are called, respectively, \textit{single and double layer potential} relative to the set $\varOmega$.
By definition, $S_{\varOmega}\varphi$ and $D_{\varOmega}\varphi$ are  well-defined and harmonic
in $\mathbb{R}^d\setminus \partial\varOmega$.
Further, still for $d\geq 3$, we have
\begin{equation*}
	S_{\varOmega}\varphi=O(|x|^{2-d})
	\qquad \textrm{and}\qquad
	D_{\varOmega}\varphi=O(|x|^{1-d})
\end{equation*}
as $|x|\rightarrow +\infty$.
In addition, if $\varphi$ has zero mean on $\partial\varOmega$, the decay rate of the single layer potential
$S_{\varOmega}$ is increased, precisely,
\begin{equation*}
	\textrm{if}\quad\int_{\partial\varOmega}\varphi(x)d\sigma(x)=0,
	\quad\textrm{then}\quad
	S_{\varOmega}\varphi=O(|x|^{1-d})
\end{equation*}
as $|x|\rightarrow +\infty$ (such property holds for $d\geq2$).

Next, we introduce the compact operator $K_{\varOmega}: L^2(\partial\varOmega)\rightarrow L^2(\partial\varOmega)$
\begin{equation}\label{singular integral K}
	K_{\varOmega}\varphi(x):=\frac{1}{\omega_d}\textrm{p.v.}
	\int_{\partial\varOmega}\frac{(y-x)\cdot n_y}{|x-y|^d}\varphi(y)d\sigma(y)
\end{equation}
and its $L^2-$adjoint 
\begin{equation}\label{singular integral K*}
	K^*_{\varOmega}\varphi(x)=\frac{1}{\omega_d}\textrm{p.v.}
	\int_{\partial\varOmega}\frac{(x-y)\cdot n_x}{|x-y|^d}\varphi(y)d\sigma(y).
\end{equation}
Given a function $v$ defined in a neighbourhood of $\partial\varOmega$, set
\begin{equation*}
	v(x)\big|_{\pm}:=\lim\limits_{h\rightarrow 0^+}v(x\pm h n_x),\ \  x\in\partial\varOmega.
\end{equation*}
The following relations hold, a.e. in $\partial\varOmega$,
\begin{equation}\label{traces layer potentials}
	S_{\varOmega}\varphi\big|_{+}=S_{\varOmega}\varphi\big|_{-},\quad
	\frac{\partial S_{\varOmega}\varphi}{\partial n_x}\Big|_{\pm}=\left(\pm\tfrac{1}{2}I+K^*_{\varOmega}\right)\varphi,\quad
	D_{\varOmega}\varphi\Big|_{\pm}=\left(\mp\tfrac{1}{2}I+K_{\varOmega}\right)\varphi.
\end{equation}
For the proof see \cite{Ammari-Kang, Folland}.

\subsection*{Well-posedness}

Proving existence and uniqueness results for unbounded domains with unbounded boundary is, in general, much more difficult with respect to the case of exterior domains. The main obstacle is the control of both solution decay and integrability on the boundary and usual approach is based on the use of weighted Sobolev spaces \cite{Amrouche-Bonzom}. Here, we take advantage of the symmetry property of the half-space to extend the problem to the whole space and to establish well-posedness resorting in a standard Sobolev setting. 

Given a bounded Lipschitz domain $C\subset\mathbb{R}^d_-$ and the function $g:\partial C\rightarrow \mathbb{R}$, we define
\begin{equation*}
\widetilde{C}:=\{(x',x_d):\ (x',-x_d)\in C\}
\end{equation*}
and $G: \partial C\cup \partial\widetilde{C}\rightarrow \mathbb{R}$ as
\begin{equation*}
	G(x):=\begin{cases}
		g(x) & \text{if}\ x\in \partial C\\
		g(\tilde{x}) & \text{if}\ x\in \partial \widetilde{C}.
		\end{cases}
\end{equation*}

\begin{theorem}
The problem \eqref{direct problem} has a unique solution.
This solution coincides with the restriction to the half-space $\mathbb{R}^d_-$ of the solution to 
\begin{equation}\label{extended problem}
	\Delta U=0\quad \textrm{in}\ \mathbb{R}^d\setminus\left(C\cup \widetilde{C}\right),\quad
	\frac{\partial U}{\partial n}=G\quad \textrm{on}\ \partial C \cup \partial \widetilde{C},\quad
	U\to 0\quad \text{as}\ |x|\rightarrow +\infty.
\end{equation}
\end{theorem}

\begin{proof}
The proof is divided into three steps: uniqueness for \eqref{extended problem},
existence for \eqref{extended problem}, equivalence between \eqref{direct problem} and \eqref{extended problem}. 

1. For $\varLambda:=C\cup\widetilde{C}$, let $R>0$ be such that $\varLambda\subset B_R(0)$
and set $\varOmega_R:=B_R\setminus \varLambda$.
Given two solutions, $U_1$ and $U_2$, to problem \eqref{extended problem},
the difference $W:=U_1-U_2$ solves the corresponding homogeneous problem.
Multiplying equation $\Delta W=0$ by $W$ and integrating over the domain $\varOmega_R=B_R\setminus \varLambda$, we infer
\begin{equation*}
	\begin{aligned}
	0&=\int_{\varOmega_R}W(x)\Delta W(x) dx\\
		&=\int_{\partial B_R(0)}W(x)\frac{\partial}{\partial n}W(x)d\sigma(x)-\int_{\varOmega_R}\big|\nabla W(x)\big|^2 dx,
\end{aligned}
\end{equation*}
using integration by parts and boundary conditions. 
Exploiting the behaviour of the harmonic functions in exterior domains, as described in \eqref{asympotic behaviour},
we get
\begin{equation*}
	\Big|\int_{\partial B_R(0)}W(x)\frac{\partial}{\partial n}W(x)\,d\sigma(x)\Big|\leq \frac{C}{R^{d-2}}.
\end{equation*}
Then, as $R\to\infty$, we find
\begin{equation*}
	\int_{\mathbb{R}^d\setminus \varLambda}\big|\nabla W(x)\big|^2 dx=0
\end{equation*}
which implies $W=0$.

2. We represent the solution of \eqref{extended problem} by means of single layer potential
\begin{equation}
	S_{\varLambda}\psi(x)=\int_{\partial \varLambda}\varGamma(x-y)\psi(y)d\sigma(y),
	\qquad x\in\mathbb{R}^d\setminus \varLambda,
\end{equation}
with function $\psi$ to be determined. By the properties of single layer potential, $S_{\varLambda}\psi$ is harmonic
in $\mathbb{R}^d\setminus\partial \varLambda$, $S_{\varLambda}\psi(x)$=$\textit{O}(|x|^{2-d})$ as $|x|\rightarrow\infty$
and we have
\begin{equation*}
	\frac{\partial S_{\varLambda}\psi}{\partial n}(x)\bigg|_+=\tfrac{1}{2}\psi+K^{*}_\varLambda\psi,
	\qquad x\in \partial\varLambda.
\end{equation*}
We now prove that there exists a function $\psi$ such that
\begin{equation}\label{boundary integral equation for unbounded domain}
	\left(\tfrac{1}{2}I+K^*_{\varLambda}\right)\psi(x)=G(x),
	\qquad x\in \partial\varLambda.
\end{equation} 
To this aim we prove that the operator $\tfrac{1}{2}I+K^*_\varLambda$ is injective.
Indeed, given $\zeta\in\ker(\frac{1}{2}I+K^*_\varLambda)$, define
 $V:=S_{\varLambda}\zeta$.
Then, from the properties of layer potentials, $W$ solves
\begin{equation*}
	\Delta V=0\quad \text{in}\;\mathbb{R}^d\setminus \varLambda,\qquad
	\frac{\partial}{\partial n}V=0\quad\text{on}\;\partial \varLambda,\qquad
	V\rightarrow 0\quad\text{as}\;|x|\rightarrow \infty,
\end{equation*}
hence, from Step 1., the function $V$ is identically zero; then it follows that $\zeta\equiv 0$.
Finally,  observing that $K^*_{\varLambda}$ is a compact operator and $G\in L^2(\partial \varLambda)$, 
equation \eqref{boundary integral equation for unbounded domain} admits a unique solution.

3. To show that $u:=U\big|_{x_d<0}$, we have to verify that the normal derivative is null on the boundary of the half-space.
This is an immediate consequence of the symmetry property
\begin{equation}\label{symmetry of U}
	U(x',-x_d)=U(x',x_d),
\end{equation}
which follows from the observation that, by definition of the boundary datum $G$, the function
$\bar{u}(x',x_d):=U(x',-x_d)$ solves \eqref{extended problem} and the solution to such problem is unique.

As a consequence of \eqref{symmetry of U}, we obtain
\begin{equation*}
	\frac{\partial \bar{u}}{\partial x_d}(x',x_d)=\frac{\partial U}{\partial x_d}(x',x_d)
	=-\frac{\partial U}{\partial x_d}(x',-x_d).
\end{equation*} 
Thus the derivative of $U$ with respect to $x_d$ computed at any point with $x_d=0$ is zero.
\end{proof}

\subsection*{Representation formula}

Next, we derive an integral representation formula for the solution $u$ to problem (\ref{direct problem}).
This makes use of the single and double layer potentials defined in (\ref{single and double layer definition}) and of contributions relative to the image cavity $\widetilde{C}$, given by
\begin{equation}\label{single double layer potential image}
	\begin{aligned}
	\widetilde{S}_C \varphi(x)&:=\int_{\partial C}\varGamma(\widetilde{x}-y)\varphi(y)d\sigma(y),
		\qquad & x\in\mathbb{R}^d, \\
	\widetilde{D}_C \varphi(x)&:=\int_{\partial C}\frac{\partial}{\partial n_y}\varGamma(\widetilde{x}-y)\varphi(y)d\sigma(y)
		\qquad & x\in\mathbb{R}^d\setminus \partial\widetilde{C}.
	\end{aligned}
\end{equation}
These operators, referred to as {\it image layer potentials}, can be read as single and double layer potentials
on $\widetilde{C}$ applied to the reflection of the function $\varphi$ with respect to $x_d$ coordinate. 

\begin{theorem} \label{theorem representation formula}
The solution $u$ to problem \eqref{direct problem} is such that
\begin{equation}\label{representation formula}
	u(x)=S_Cg(x)-D_Cf(x)+\widetilde{S}_Cg(x)-\widetilde{D}_Cf(x),
	\qquad x\in\mathbb{R}_-^d\setminus C,
\end{equation}
where $S_C, D_C$ are defined in \eqref{single and double layer definition}, $\widetilde{S}_C, \widetilde{D}_C$
in \eqref{single double layer potential image}, $g$ is the Neumann boundary condition in \eqref{direct problem}
and $f$ is the trace of $u$ on $\partial C$. 
\end{theorem}

Using properties of layer potentials, from equation \eqref{representation formula}, we infer
\begin{equation*}
	f(x)=S_Cg(x)-\left(-\tfrac{1}{2}I+K_C\right)f(x)-\widetilde{D}_Cf(x)+\widetilde{S}_Cg(x), \quad  x\in\partial C,
\end{equation*}
where $K_C$ is defined in \eqref{singular integral K}.
Thus, the trace $f$ satisfies the relation
\begin{equation*}
	\left(\tfrac{1}{2}I+K_C+\widetilde{D}_C\right)f=S_Cg+\widetilde{S}_Cg,
\end{equation*}
which will turn out to be useful in the sequel.

Before proving Theorem \ref{theorem representation formula}, we first recall the definition of the
Neumann function, see \cite{Hein}, that is the solution $N=N(x,y)$ to
\begin{equation*}
	\Delta _y N(x,y)=\delta_x(y)\quad \textrm{in}\ \mathbb{R}^d_-,\qquad
	\frac{\partial N}{\partial y_d}(x,y)=0\quad \text{on} \ \mathbb{R}^{d-1},
\end{equation*} 
where $\delta_x(y)$ is the delta function centred in a fixed point $x\in\mathbb{R}^d$ and $\partial N/\partial y_d$
represents the normal derivative on the boundary of the half-space $\mathbb{R}^d_-$.  
The classical method of images provides the explicit expression
\begin{equation*}
	N(x,y)=\frac{\kappa_d}{|x-y|^{d-2}}+\frac{\kappa_d}{|\widetilde{x}-y|^{d-2}}.
\end{equation*}
With the function $N$ at hand, the representation formula \eqref{representation formula}
can be equivalently written as  
\begin{equation}\label{expression of representation formula with Neumann}
	\begin{aligned}
	u(x)&=\mathcal{N}(f,g)(x)\\
		&:=\int_{\partial C}\left[N(x,y)g(y)-\frac{\partial}{\partial n_y}N(x,y)f(y)\right]d\sigma(y),
			\quad x\in\mathbb{R}_-^d\setminus C,
	\end{aligned}
\end{equation}
which we now prove.

\begin{proof}[Proof of Theorem \ref{theorem representation formula}]
Given $R, \varepsilon >0$ such that $C\subset B_R(0)$ and $B_{\varepsilon}(x)\subset \mathbb{R}^d_-\setminus C$, let
\begin{equation*}
	\varOmega_{R,\varepsilon}:=\Bigl(\mathbb{R}^d_-\cap B_R(0)\Bigr)\setminus  \Bigl(C \cup B_{\varepsilon}(x)\Bigr).
\end{equation*}
We also define $\partial B^h_R(0)$ as the intersection of the hemisphere with the boundary of the half-space,
and with $\partial B^b_R(0)$ the spherical cap (see Figure \ref{figure hemisphere with cavity}).
\begin{figure}[h]
\centering
\def\svgwidth{220pt}
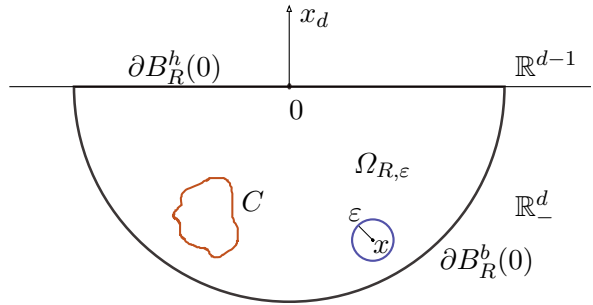
\caption{Domain $\varOmega_{R,\varepsilon}$ used to obtain the integral representation formula
\eqref{representation formula}.}\label{figure hemisphere with cavity}
\end{figure}
Applying second Green's identity to $N(x,\cdot)$ and $u$ in $\varOmega_{R,\varepsilon}$, we get
\begin{equation*}
	\begin{aligned}
	0=&\int_{\varOmega_{R,\varepsilon}}\left[N(x,y)\Delta u(y)-u(y)\Delta_y N(x,y)\right]dy\\
	=&\int_{\partial B^h_R(0)}\left[N(x,y)\frac{\partial u}{\partial y_d}(y)-\frac{\partial}{\partial y_d}N(x,y)u(y)\right]d\sigma(y)\\
	&+\int_{\partial B^b_R(0)}\left[N(x,y)\frac{\partial u}{\partial n_y}(y)-\frac{\partial}{\partial n_y}N(x,y)u(y)\right]d\sigma(y)\\
	&+\int_{\partial B_{\varepsilon}(x)}\left[\frac{\partial}{\partial n_y}N(x,y)u(y)-N(x,y)\frac{\partial u}{\partial n_y}(y)\right]d\sigma(y)\\
	&-\int_{\partial C}\left[N(x,y)\frac{\partial u}{\partial n_y}(y)-\frac{\partial}{\partial n_y}N(x,y)u(y)\right]d\sigma(y)\\
	=& I_1+I_2+I_3-\mathcal{N}(f,g)(x).
	\end{aligned}
\end{equation*}     
The term $I_1$ is zero since both the normal derivative of the function $N$ and $u$ are zero above the boundary of the half-space.

Next, taking into account the behaviour of harmonic functions in exterior domains, formulas \eqref{asympotic behaviour}, we deduce
\begin{equation*}
	\begin{aligned}
	\Big|\int_{\partial B^b_R(0)}\frac{\partial }{\partial n_y}N(x,y)u(y)\,d\sigma(y)\Big|
		&\leq \frac{C}{R^{2d-3}}\int_{\partial B^b_R(0)}\,d\sigma(y)=\frac{C}{R^{d-2}},\\
	\Big|\int_{\partial B^b_R(0)}N(x,y)\frac{\partial u(y)}{\partial n_y}\,d\sigma(y)\Big|
		&\leq \frac{C}{R^{2d-3}}\int_{\partial B^b_R(0)}\,d\sigma(y)=\frac{C}{R^{d-2}},
	\end{aligned}
\end{equation*}
where $C$ denotes a generic positive constant.
As $R\to+\infty$, $I_2$ tends to zero. 

Finally, we decompose $I_3$ as
\begin{equation*}
	I_3=I_{31}-I_{32}=\int_{\partial B_{\varepsilon}(x)}\frac{\partial}{\partial n_y}N(x,y)u(y)\,d\sigma(y)
		-\int_{\partial B_{\varepsilon}(x)}N(x,y)\frac{\partial u}{\partial n_y}(y)d\sigma(y).
\end{equation*}
Using the expression of $N$ and the continuity of $u$, we derive
\begin{equation*}
	\begin{aligned}
	I_{31}=\int_{\partial B_{\varepsilon}(x)}\frac{\partial}{\partial n_y}N(x,y)u(y)d\sigma(y)
	=& u(x)\int_{\partial B_{\varepsilon}(x)}\frac{\partial}{\partial n_y}N(x,y)d\sigma(y)\\
	&+\int_{\partial B_{\varepsilon}(x)}\left[u(y)-u(x)\right]\frac{\partial}{\partial n_y}N(x,y)d\sigma(y),
	\end{aligned}
\end{equation*} 
which tends to $u(x)$ as $\varepsilon\rightarrow 0$.
Moreover, we infer
\begin{equation*}
	\begin{aligned}
	\big|I_{32}\big|&\leq C \sup\limits_{y\in\partial B_{\varepsilon}(x)}\Big|\frac{\partial u}{\partial n_y}
		\Big|\int_{\partial B_{\varepsilon}(x)}\Big|N(x,y)\Big|d\sigma(y)\\
	&\leq C' \sup\limits_{y\in\partial B_{\varepsilon}(x)}\Big|\frac{\partial u}{\partial n_y}
		\Big|\left[\int_{\partial B_{\varepsilon}(x)}\frac{1}{\varepsilon^{d-2}}d\sigma(y)
	+\int_{\partial B_{\varepsilon}(x)}\frac{1}{|\widetilde{x}-y|^{d-2}}d\sigma(y)\right].
	\end{aligned}
\end{equation*}
Observing that both the integrals tend to zero when $\varepsilon$ goes to zero because the second one has a continuous
kernel while the first one behaves as $O(\varepsilon)$, we infer that $I_{32}\to 0$ as $\varepsilon\to 0$. 
Putting together all the results, we obtain \eqref{expression of representation formula with Neumann}.
\end{proof}

\section{Spectral analysis}\label{section spectral analysis}

Following the approach of Ammari and Kang, see \cite{Ammari-Kang, Ammari-Kang1}, in this Section, we prove
the invertibility of the operator $\frac12\,I+K_C+\widetilde{D}_C$ showing that, under suitable assumptions, the following inclusion holds
\begin{equation*}
	\sigma(K_C+\widetilde{D}_C)\subset\left(-1/2,1/2\right].
\end{equation*}
Such task is accomplished by determining the spectrum of the adjoint operator $K^*_C+\widetilde{D}^*_C$ in $L^2(\partial C)$,
relying on the fact that the two spectra are conjugate.

The explicit expression of $K^*_C$ is in \eqref{singular integral K*}. 
Computing the $L^2$-adjoint of $\widetilde{D}_C$ is straightforward: indeed, given $\psi\in L^2(\partial C)$, we have
\begin{equation*}
\begin{aligned}
\int_{\partial C}\psi(x)\widetilde{D}_C\varphi(x)d\sigma(x)
&=\int_{\partial C}\psi(x)\left(\frac{1}{\omega_d}\int_{\partial C}\frac{(y-\widetilde{x})\cdot n_y}{|\widetilde{x}-y|^d}\varphi(y)d\sigma(y)\right)d\sigma(x)\\
&=\int_{\partial C}\varphi(y)\left(\frac{1}{\omega_d}\int_{\partial C}\frac{(y-\widetilde{x})\cdot n_y}{|\widetilde{x}-y|^d}\psi(x)d\sigma(x)\right)d\sigma(y)
\end{aligned}
\end{equation*}
and thus
\begin{equation}\label{adjoint D tilde}
	\widetilde{D}^*_C\varphi(x)=\frac{1}{\omega_d}\int_{\partial C}\frac{(x-\widetilde{y})\cdot n_x}{|\widetilde{y}-x|^d}\varphi(y)d\sigma(y).
\end{equation}
Note that the kernel of the integral operator $\widetilde{D}^*_C$ is smooth on $\partial C$.

As proved in \cite{Kellogg}, for smooth domains the eigenvalues of $K^*_C$ on $L^2(\partial C)$ lie in $(-1/2,1/2]$;
the same result it is also true for Lipschitz domain (see \cite{Ammari-Kang, Escauriaza-Fabes-Verchota}). 
With the same approach, it can be shown that the same property holds true for $K^*_C+\widetilde{D}^*_C$.

\begin{theorem}\label{theorem operator spectrum}
Let $C$ be an open bounded domain with Lipschitz boundary. Then
\begin{equation*}
	\sigma(K^*_C+\widetilde{D}^*_C)\subset\left(-1/2,1/2\right].
\end{equation*} 
\end{theorem}

\noindent
For completeness, we provide here a complete proof of such fact.

Firstly, we observe that the regular operator $\widetilde{D}^*_C$ on the boundary of the cavity can be seen
as the normal derivative of an appropriate single layer potential. 

\begin{lemma}
Given $\varphi\in L^2(\partial C)$ we have that
\begin{equation*}
	\widetilde{D}^*_C\varphi (x)=\frac{\partial}{\partial n_x}\left(S_{\widetilde{C}}\widetilde{\varphi}(x)\right),\ \ x\in\partial C,
\end{equation*}
where $\widetilde{\varphi}\in L^2(\partial\widetilde{C})$ is defined by $\widetilde{\varphi}(x):=\varphi(\widetilde{x})$.
\end{lemma}  

\begin{proof}
Using the expression \eqref{adjoint D tilde} of $\widetilde{D}^*_C$ and the identity
\begin{equation*}
	\nabla_x\left(\frac{1}{(2-d)|x-y|^{d-2}}\right)=\frac{x-y}{|x-y|^d},
\end{equation*}  
we find that
\begin{equation*}
	\widetilde{D}^*_C\varphi(x)=\nabla_x\left(\int_{\partial C}
	\frac{\kappa_d\,\varphi(y)}{|\widetilde{y}-x|^{d-2}}d\sigma(y)\right)\cdot n_x.
\end{equation*}
Given $\varphi\in L^2(\partial C)$ and $\widetilde{\varphi}\in L^2(\partial \widetilde{C})$ as previously defined, we have
\begin{equation*}
	\begin{aligned}
	\int_{\partial C}\frac{\varphi(y)}{|\widetilde{y}-x|^{d-2}}d\sigma(y)
	&=\int_{\partial \widetilde{C}}\frac{\varphi(\widetilde{z})}{|\widetilde{\widetilde{z}}-x|^{d-2}}d\sigma(z)\\
	&=\int_{\partial \widetilde{C}}\frac{\varphi(\widetilde{z})}{|z-x|^{d-2}}d\sigma(z)
		=\int_{\partial \widetilde{C}}\frac{\widetilde{\varphi}(z)}{|z-x|^{d-2}}d\sigma(z),
	\end{aligned}
\end{equation*} 
which gives the conclusion.
\end{proof}

We are now ready to prove the main result of this Section.

\begin{proof}[Proof of Theorem \eqref{theorem operator spectrum}]
Given $\varphi\in L^2(\partial C)$, let $\psi$ be defined by
$\psi:=S_C\varphi+S_{\widetilde{C}}\widetilde{\varphi}$.
By the known properties of single layer potentials,  we derive on $\partial C$ 
\begin{equation*}
	\frac{\partial\psi}{\partial n}\Big|_{\pm}=\left(\pm\tfrac{1}{2}I+K^*_C+\widetilde{D}^*_C\right)\varphi
\end{equation*} 
and, as a consequence, 
\begin{equation}\label{eigenfunctions in terms of W}
	\frac{\partial\psi}{\partial n}\bigg|_+ + \frac{\partial\psi}{\partial n}\bigg|_-=2\left(K^*_C+\widetilde{D}^*_C\right)\varphi,\qquad
	\frac{\partial \psi}{\partial n}\bigg|_+ - \frac{\partial \psi}{\partial n}\bigg|_-=\varphi.
\end{equation}
Taking a linear combination of the two relations in \eqref{eigenfunctions in terms of W}, we deduce
\begin{equation*}
	\begin{aligned}
	\left(\lambda I-K^*_C-\widetilde{D}^*_C\right)\varphi&=\lambda\left(\frac{\partial  \psi}{\partial n}\bigg|_+
		- \frac{\partial \psi}{\partial n}\bigg|_-\right)-\frac{1}{2}\left(\frac{\partial \psi}{\partial n}\bigg|_+
		+ \frac{\partial \psi}{\partial n}\bigg|_-\right)\\
	&=\left(\lambda-\frac{1}{2}\right)\frac{\partial \psi}{\partial n}\bigg|_+
		-\left(\lambda+\frac{1}{2}\right)\frac{\partial \psi}{\partial n}\bigg|_-.
	\end{aligned}
\end{equation*}
If $\lambda$ is an eigenvalue of $K^*_C+\widetilde{D}^*_C$ with eigenfunction $\varphi$, 
then
\begin{equation*}
	\left(\lambda-\frac{1}{2}\right)\frac{\partial \psi}{\partial n}\bigg|_+
		-\left(\lambda+\frac{1}{2}\right)\frac{\partial \psi}{\partial n}\bigg|_-=0,
	\qquad\textrm{on}\quad\partial C.
\end{equation*}
Multiplying such relation by the function $\psi$ and integrating over $\partial C$, we get
\begin{equation}\label{second equation eigenvalue problem}
	\left(\lambda-\frac{1}{2}\right)\int_{\partial C} \psi(x)\frac{\partial \psi}{\partial n}(x)\bigg|_+d\sigma(x)
		-\left(\lambda+\frac{1}{2}\right)\int_{\partial C} \psi(x)\frac{\partial \psi}{\partial n}(x)\bigg|_-d\sigma(x)=0.
\end{equation}
Integrating by parts we have
\begin{equation}\label{integral of interior}
	\begin{aligned}
	\int_{\partial C} \psi(x)\frac{\partial \psi}{\partial n}(x)\bigg|_-d\sigma(x)
		&=\int_{C} \psi(x)\Delta \psi(x)dx+\int_{C}\big|\nabla \psi(x)\big|^2\,dx\\
		&=\int_{C}\big|\nabla \psi(x)\big|^2\,dx.
	\end{aligned}
\end{equation}
The first integral in \eqref{second equation eigenvalue problem} can be dealt with as done in 
the proof of Theorem \ref{theorem representation formula}.
Precisely, given large $R>0$, applying the Green's formula in $\varOmega_R:=\left(\mathbb{R}^d_-\cap B_R(0)\right)\setminus C$,
we get
\begin{equation*}
\begin{aligned}
\int_{\partial C} \psi(x)\frac{\partial \psi}{\partial n}(x)\bigg|_+d\sigma(x)&=\int_{\partial B^h_R(0)}\psi(x)\frac{\partial \psi}{\partial x_d}(x)d\sigma(x)
+\int_{\partial B^b_R(0)}\psi(x)\frac{\partial \psi}{\partial n}(x)\bigg|_+d\sigma(x)\\
&\quad -\int_{\varOmega_R} \psi(x)\Delta \psi(x)dx-\int_{\varOmega_R}\big|\nabla \psi(x)\big|^2dx,
\end{aligned}
\end{equation*}
where $\partial B^h_R(0)$ is the intersection of the hemisphere
with the half-space and $\partial B^b_R(0)$ is the spherical cap.
The quantity $\partial \psi/\partial n$ is identically zero on the boundary of the half-space since the kernel of the operator is the normal derivative of the Neumann function which, by hypothesis, is null on $\mathbb{R}^{d-1}$. Moreover, $\psi$ is harmonic in $\varOmega_R$, so we infer
\begin{equation*}
\int_{\partial C} \psi(x)\frac{\partial \psi}{\partial n}(x)\bigg|_+d\sigma(x)=\int_{\partial B^b_R(0)}\psi(x)\frac{\partial \psi}{\partial n}(x)\bigg|_+d\sigma(x)-\int_{\varOmega_R}\big|\nabla \psi(x)\big|^2dx.
\end{equation*}
Recalling the asymptotic behaviour of simple layer potential,
\begin{equation*}
	\big|S_C\varphi\big|+\big|S_{\widetilde{C}}\varphi\big|=O(|x|^{2-d}),\qquad
	\big|\nabla S_C\varphi\big|+\big|\nabla S_{\widetilde{C}}\varphi\big|=O(|x|^{1-d})
	\qquad \textrm{as}\ \ |x|\to\infty.
\end{equation*}
we obtain, for some $C>0$,
\begin{equation*}
	\begin{aligned}
	\bigg|\int_{\partial B^b_R(0)}\psi(x)\frac{\partial \psi}{\partial n}(x)\bigg|_+d\sigma(x)\bigg|
	&\leq \int_{\partial B^b_R(0)}\Big|\psi(x)\Big|\Big|\frac{\partial \psi}{\partial n}(x)\bigg|_+\Big|d\sigma(x)\\
	&\leq\frac{C}{R^{2d-3}}\int_{\partial B^b_R(0)}d\sigma(x)=\frac{1}{R^{d-2}}.
	\end{aligned}
\end{equation*}  
Passing to the limit $R\to+\infty$, we find
\begin{equation}\label{integral of exterior}
	\int_{\partial C} \psi(x)\frac{\partial \psi}{\partial n}\bigg|_+d\sigma(x)
		=-\int_{\mathbb{R}^d_-\setminus \overline{C}}\big|\nabla \psi(x)\big|^2dx.
\end{equation}
Plugging \eqref{integral of interior} and \eqref{integral of exterior} into \eqref{second equation eigenvalue problem},
we infer the identity
\begin{equation*}
	\left(\lambda-\frac{1}{2}\right)\int_{\mathbb{R}^d_-\setminus C}\big|\nabla \psi(x)\big|^2dx
	+\left(\lambda+\frac{1}{2}\right)\int_{C}\big|\nabla \psi(x)\big|^2dx=0,
\end{equation*}
that is
\begin{equation*}
	(A+B)\lambda=\frac{1}{2}(A-B)
\end{equation*}
with
\begin{equation*}
	A:=\int_{\mathbb{R}^d_-\setminus C}\big|\nabla \psi(x)\big|^2dx
	\qquad\textrm{and}\qquad
	B:=\int_{C}\big|\nabla \psi(x)\big|^2dx.
\end{equation*}
The coefficient of $\lambda$ is non-zero.
On the contrary,  if $A+B=0$ then $\nabla \psi=0$ in $\mathbb{R}^d_-$ which means that $\psi\equiv 0$,
hence, from the second equation in \eqref{eigenfunctions in terms of W}, we  get $\varphi=0$ in $\partial C$.

Therefore, solving with respect to $\lambda$, we finally get
\begin{equation}\label{deco lambda}
	\lambda=\frac{1}{2}\cdot\frac{A-B}{A+B}\in\left[-\frac{1}{2},\frac{1}{2}\right].
\end{equation}
The value $\lambda=-1/2$ is not an eigenvalue for the operator $K_C^*+\widetilde{D}_C^*$.
Indeed, in such a case, we would have
\begin{equation*}
	A=\int_{\mathbb{R}^d_-\setminus C}\big|\nabla \psi(x)\big|^2\,dx=0,
\end{equation*} 
and thus $\psi=0$ in $\mathbb{R}^d_-\setminus C$. 
By definition of $\psi$, we deduce that $\psi=0$
on $\partial C$ and since  $\psi$ is harmonic in $C$, we get that $\psi=0$ also in $C$.
As before, by \eqref{eigenfunctions in terms of W}, this would imply that $\varphi=0$ in $\partial C$.
\end{proof}

For completeness, let us observe that the value $\lambda=1/2$ is an eigenvalue with geometric multiplicity equal to one.
Indeed, identity \eqref{deco lambda} implies that, for such value of $\lambda$, 
\begin{equation*}
	B=\int_{C}\big|\nabla \psi(x)\big|^2 dx=0,
\end{equation*}
hence $\psi$ is constant in $C$. Normalizing $\psi=1$ in $C$, the function $\psi$ in $\mathbb{R}^d_-\setminus C$ is given by the restriction of the solution $U$ to the Dirichlet problem in the exterior domain $\mathbb{R}^d\setminus \left(C\cup \widetilde{C}\right)$ with boundary data equal to 1. Then, by the second equation in \eqref{eigenfunctions in terms of W}, the function $\varphi$ is the normal derivative of $U$ at $\partial C$.

\section{Asymptotic expansion}\label{section asymptotic expansion}

In this Section, we derive an asymptotic formula for the solution of  problem \eqref{direct problem} when the cavity $C$
is small compared to the distance from the half-space $\mathbb{R}^{d-1}$.
For the reader's convenience, we recall that the cavity $C$ has the structure
\begin{equation*}
	C_{\varepsilon}:=C=z+\varepsilon B
\end{equation*}
where $B$ is a bounded Lipschitz set containing the origin. Moreover, we assume that
\begin{equation}\label{distance from half-space}
\textrm{dist}(z,\mathbb{R}^{d-1})\geq \delta_0>0
\end{equation}
otherwise, for the application we have in mind, the problem does not have a real physical meaning.
To emphasize the dependence of the solution to the direct problem by the parameter $\varepsilon$ we denote it by $u_{\varepsilon}$.
For brevity, we denote the layer potentials relative to $C_{\varepsilon}$ by the index $\varepsilon$, {\it viz.}
\begin{equation*}
	S_{\varepsilon}=S_{C_{\varepsilon}},\quad
	D_{\varepsilon}=D_{C_{\varepsilon}},\quad
	\widetilde{S}_{\varepsilon}=\widetilde{S}_{C_{\varepsilon}},\quad
	\widetilde{D}_{\varepsilon}=\widetilde{D}_{C_{\varepsilon}},\quad
	K_{\varepsilon}=K_{C_{\varepsilon}}
\end{equation*}
and the trace of the solution $u_{\varepsilon}$ on $\partial C_{\varepsilon}$ by $f_{\varepsilon}$. In this way the representation formula \eqref{representation formula} reads as
\begin{equation*}
	u_{\varepsilon}=S_{\varepsilon}g-D_{\varepsilon}f_{\varepsilon}-\widetilde{D}_{\varepsilon}f_{\varepsilon}+\widetilde{S}_{\varepsilon}g.
\end{equation*}
At $x\in\mathbb{R}^{d-1}$, taking into account that $x=\widetilde{x}$, it follows that
\begin{equation*}
	\begin{aligned}
	S_{\varepsilon}g(x)&=\int_{\partial C_{\varepsilon}}\varGamma(x-y)g(y)\,d\sigma(y)
		=\int_{\partial C_{\varepsilon}}\varGamma(\widetilde{x}-y)g(y)\,d\sigma(y)
		=\widetilde{S}_{\varepsilon}g(x)\\
	D_{\varepsilon}f_{\varepsilon}(x)&=\int_{\partial C_{\varepsilon}}\frac{\partial}{\partial n_y}\varGamma(x-y)f_{\varepsilon}(y)\,d\sigma(y)
		=\int_{\partial C_{\varepsilon}}\frac{\partial}{\partial n_y}\varGamma(\widetilde{x}-y)f_{\varepsilon}(y)\,d\sigma(y)
		=\widetilde{D}_{\varepsilon}f_{\varepsilon}(x)
	\end{aligned}
\end{equation*}
Hence, we obtain the equality
\begin{equation*}
	\tfrac12\,u_{\varepsilon}(x)=S_{\varepsilon}g(x)-D_{\varepsilon}f_{\varepsilon}(x),
	\qquad  x\in\mathbb{R}^{d-1}.
\end{equation*} 
Associating with the relation at the boundary $\partial C_{\varepsilon}$ and by \eqref{theorem operator spectrum}]
\begin{equation*}
	f_{\varepsilon}(x)=\left(\tfrac{1}{2}I+K_{\varepsilon}+\widetilde{D}_{\varepsilon}\right)^{-1}
		\left(S_{\varepsilon}g+\widetilde{S}_{\varepsilon}g\right)(x),
		\qquad x\in\partial C_{\varepsilon},
\end{equation*}
we get the identity
\begin{equation}\label{representation formula for asymptotic expansion}
	\tfrac{1}{2}u_{\varepsilon}(x)=J_1(x)+J_2(x),\qquad x\in\mathbb{R}^{d-1},
\end{equation}
where
\begin{equation*}
	\begin{aligned}
	J_1(x)&:=\int_{\partial C_{\varepsilon}}\varGamma(x-y)g(y)\,d\sigma(y),\\
	J_2(x)&:=-\int_{\partial C_{\varepsilon}}\frac{\partial}{\partial n_y}\varGamma(x-y)\left(\tfrac{1}{2}I+K_{\varepsilon}
		+\widetilde{D}_{\varepsilon}\right)^{-1}\left(S_{\varepsilon}g+\widetilde{S}_{\varepsilon}g\right)(y)\,d\sigma(y).\\
	\end{aligned}
\end{equation*}
Analyzing in details the dependence with respect to $\varepsilon$ of such relation, we obtain an explicit expression
for the first two terms in the asymptotic expansion of $u_{\varepsilon}$ at $\mathbb{R}^{d-1}$.

In what follows, for any fixed value of $\varepsilon>0$, given $h\,:\,\partial C_\varepsilon\to\mathbb{R}$,
we introduce the function $\widehat{h}\,:\,\partial B\to\mathbb{R}$ defined by
\begin{equation*}
	\widehat{h}(\zeta;\varepsilon):=h(z+\varepsilon\,\zeta),\qquad \zeta\in \partial B.
\end{equation*} 
This definition is useful to consider integrals over a set that is independent on $\varepsilon$.

\begin{theorem} \label{theorem asymptotic expansion}
Let us assume \eqref{distance from half-space}. For any $\varepsilon>0$, let $g\in L^2(\partial C_\varepsilon)$ such that $\widehat{g}$ is independent on $\varepsilon$.
Then, at any $x\in \mathbb{R}^{d-1}$, the following expansion holds
\begin{equation}\label{asymptotic expansion}
	\begin{aligned}
	u_{\varepsilon}(x)
	&=2\varepsilon^{d-1}\varGamma(x-z)\int_{\partial B}\widehat{g}(\zeta)d\sigma(\zeta)\\
	&+2\varepsilon^d\nabla\varGamma(x-z)\cdot\int_{\partial B}\left\{n_{\zeta}\left(\tfrac{1}{2}I+K_B\right)^{-1}S_B\widehat{g}(\zeta)
		-\zeta\widehat{g}(\zeta)\right\}d\sigma(\zeta)+O(\varepsilon^{d+1}),
	\end{aligned}
\end{equation}
where $O(\varepsilon^{d+1})$ denotes a quantity uniformly bounded by $C\varepsilon^{d+1}$ with $C=C(\delta_0)$ which tends to infinity when $\delta_0$ goes to zero. 
\end{theorem}
To prove the theorem we first show the following expansion for the operator $\left(\tfrac{1}{2}I+K_{\varepsilon}+\widetilde{D}_{\varepsilon}\right)^{-1}$

\begin{lemma}
We have
\begin{equation}\label{some expansion}
	\left(\tfrac{1}{2}I+K_{\varepsilon}
	+\widetilde{D}_{\varepsilon}\right)^{-1}\left(S_{\varepsilon}g+\widetilde{S}_{\varepsilon}g\right)(z+\varepsilon \zeta)
	=\varepsilon\left(\tfrac{1}{2}I+K_B\right)^{-1}S_B\widehat{g}(\zeta)+O(\varepsilon^{d-1})
\end{equation}
\end{lemma}

\begin{proof}
We analyse, separetely, the terms $\left(\tfrac{1}{2}I+K_{\varepsilon}+\widetilde{D}_{\varepsilon}\right)^{-1}$ and $S_{\varepsilon}+\widetilde{S}_{\varepsilon}$, collecting, at the very end, the corresponding expansions.

Since $K_{\varepsilon}+\widetilde{D}_{\varepsilon}$ is compact and its spectrum is contained in $(-1/2,1/2]$ there exists $\delta>0$ such that
\begin{equation*}
\sigma\left(K_{\varepsilon}+\widetilde{D}_{\varepsilon}\right)\subset (-1/2+\delta,1/2].
\end{equation*} 
Then, the operator 
\begin{equation*}
A_{\varepsilon}:=\tfrac{1}{2}I-K_{\varepsilon}-\widetilde{D}_{\varepsilon}
\end{equation*}
is such that $\sigma\left(A_{\varepsilon}\right)\subset [0, 1-\delta)$ and thus has spectral radius strictly smaller than 1.
As a consequence, taking the powers of the operator $A_{\varepsilon}$ one finds
\begin{equation}\label{estimates A}
\|A^h_{\varepsilon}\|\leq 1\quad \forall h\quad\textrm{and}\quad \|A^{h_0}_{\varepsilon}\|< 1 \quad \textrm{for some $h_0$}.
\end{equation}
The inverse operator of  $I-A_{\varepsilon}=\frac{1}{2}I+K_{\varepsilon}+\widetilde{D}_{\varepsilon}$ can be represented by the Neumann series that is
\begin{equation*}
	\left(I-A_{\varepsilon}\right)^{-1}=\sum_{h=0}^{+\infty}A^h_{\varepsilon}=\sum_{h=0}^{+\infty}\left(\tfrac{1}{2}I-K_{\varepsilon}-\widetilde{D}_{\varepsilon}\right)^h.
\end{equation*}
At the point $z+\varepsilon \zeta$, we obtain
\begin{equation*}
	\begin{aligned}
	&\left(K_{\varepsilon}+\widetilde{D}_{\varepsilon}\right)\varphi(z+\varepsilon \zeta)\\
	&\qquad=\frac{1}{\omega_d}\textrm{p.v.}\int_{\partial C_{\varepsilon}}
		\frac{(y-z-\varepsilon\zeta)\cdot n_y}{|z+\varepsilon\zeta-y|^d}\varphi(y)\,d\sigma(y)
		+\int_{\partial C_{\varepsilon}}\frac{\partial}{\partial n_y} \varGamma(\tilde z+\varepsilon\tilde \zeta-y)\varphi(y)\,d\sigma(y)\\
	&\qquad =\frac{1}{\omega_d}\textrm{p.v.}\int_{\partial B}\frac{(\eta-\zeta)\cdot n_{\eta}}{|\zeta-\eta|^d}\widehat{\varphi}(\eta)\,d\sigma(\eta)
		+ \varepsilon^{d-1}\int_{\partial B}
		\frac{\partial}{\partial n_{\eta}}\varGamma(\widetilde{z}+\varepsilon\widetilde{\zeta}-z-\varepsilon\eta)
		\hat \varphi(\eta)\,d\sigma(\eta) \\
	&\qquad =K_{B}\widehat{\varphi}(\zeta)
		+\varepsilon^{d-1}R_{\varepsilon}\widehat{\varphi}(\zeta),	\end{aligned}
\end{equation*}
where 
\begin{equation*}
	R_{\varepsilon}\widehat{\varphi}(\zeta):=\int_{\partial B}\frac{\partial}{\partial n_{\eta}}
		\varGamma\left(\widetilde{z}-z+\varepsilon(\widetilde{\zeta}-\eta)\right)\widehat{\varphi}(\eta)d\sigma(\eta)
\end{equation*}
is uniformly bounded in $\varepsilon$. Using these results, we calculate $A_{\varepsilon}^h$ highlighting the term that do not contain $\varepsilon$ and the one of order $d-1$, that is
\begin{equation*}
A_{\varepsilon}^h=\left(\tfrac{1}{2}I-K_B\right)^h- \varepsilon^{d-1}E_{h,\varepsilon}
\end{equation*}
where
\begin{equation*}
E_{h,\varepsilon}=\sum_{j=1}^{h}A_{\varepsilon}\cdots A_{\varepsilon}\underbrace{R_{\varepsilon}}_{j-th}A_{\varepsilon}\cdots A_{\varepsilon}.
\end{equation*}
For $h_0$ as in \eqref{estimates A} and $h>h_0$ we have
\begin{equation*}
\|E_{h,\varepsilon}\|\leq \|R_{\varepsilon}\| \|A_{\varepsilon}\|^{2h_0} \|A_{\varepsilon}^{h_0}\|^{\left[h/h_0\right]-1}\leq  \|R_{\varepsilon}\| \|A_{\varepsilon}\|^{2h_0} \|A_{\varepsilon}^{h_0}\|^{h/h_0-1}, 
\end{equation*}
where $[\,\cdot\,]$ denotes the integer part, and thus
\begin{equation*}
\sum_{h=0}^{+\infty}\|E_{h,\varepsilon}\|\leq C \sum_{h=0}^{+\infty}\|A_{\varepsilon}^{h_0}\|^{h/h_0}
\end{equation*} 
giving the absolute convergence of $\sum E_{h,\varepsilon}$.
Summarizing we conclude that
\begin{equation}\label{exp inverse operator}
\left(I-A_{\varepsilon}\right)^{-1}=\left(\tfrac{1}{2}I+K_B\right)^{-1}+O(\varepsilon^{d-1}).
\end{equation} 
Let us evaluate the term $S_{\varepsilon}+\widetilde{S}_{\varepsilon}$.
We have
\begin{equation*}
\begin{aligned}
S_{\varepsilon}g(z+\varepsilon\zeta)&=\int_{\partial C_{\varepsilon}}\varGamma(z+\varepsilon\zeta-y)g(y)d\sigma(y)\\
&=\varepsilon\int_{\partial B}\varGamma(\zeta-\theta)\widehat{g}(\theta)d\sigma(\theta)=\varepsilon S_{B}\widehat{g}(\zeta)
\end{aligned}
\end{equation*}
and
\begin{equation*}
\begin{aligned}
\widetilde{S}_{\varepsilon}g(z+\varepsilon\zeta)&=\int_{\partial C_{\varepsilon}}\varGamma\left(\widetilde{z}+\varepsilon\widetilde{\zeta}-y\right)g(y)d\sigma(y)\\
&=\varepsilon^{d-1}\int_{\partial B}\varGamma\left(\widetilde{z}-z+\varepsilon(\widetilde{\zeta}-\theta)\right)\widehat{g}(\theta)d\sigma(\theta)\\
&=\varepsilon^{d-1}\varGamma(\widetilde{z}-z)\int_{\partial B}\widehat{g}(\theta)d\sigma(\theta)+O(\varepsilon^d)
\end{aligned}
\end{equation*}
where we have used the zero order expansion for $\varGamma$.

Collecting we infer 
\begin{equation*}
\left(S_{\varepsilon}g+\widetilde{S}_{\varepsilon}g\right)(z+\varepsilon \zeta)=\varepsilon S_B\widehat{g}(\zeta)+O(\varepsilon^{d-1})
\end{equation*}
and combining with \eqref{exp inverse operator} we obtain the conclusion.
\end{proof}

\begin{proof}[Proof of Theorem \ref{theorem asymptotic expansion}]
To prove  \eqref{asymptotic expansion}, we analyse the two integrals $J_1$ and $J_2$
in \eqref{representation formula for asymptotic expansion}.

For $x,\zeta\in\mathbb{R}^{d}$ with $x\neq 0$ and $\varepsilon$ sufficiently small, we have
\begin{equation*}
	\varGamma(x-\varepsilon\zeta)=\varGamma(x)-\varepsilon\,\nabla\varGamma(x)\cdot\zeta+O(\varepsilon^2).
\end{equation*}
Hence, for $x\in\mathbb{R}^{d-1}$, we get
\begin{equation}\label{expansion of S}
	\begin{aligned}
	J_1&=\varepsilon^{d-1}\int_{\partial B}\varGamma(x-z-\varepsilon\zeta)\,\widehat{g}(\zeta)\,d\sigma(\zeta)\\
	&=\varepsilon^{d-1}\varGamma(x-z)\int_{\partial B}\widehat{g}(\zeta)\,d\sigma(\zeta)
		-\varepsilon^{d}\nabla\varGamma(x-z)\cdot\int_{\partial B}\zeta\,\widehat{g}(\zeta)\,d\sigma(\zeta)
		+O(\varepsilon^{d+1}).
	\end{aligned}
\end{equation}
Next we consider the second integral
in \eqref{representation formula for asymptotic expansion}, written as
\begin{equation*}
	J_2=-\varepsilon^{d-1}\int_{\partial B}\frac{\partial}{\partial n_{\zeta}}\varGamma(x-z-\varepsilon\zeta)
		\widehat{h}_{\varepsilon}(\zeta)\,d\sigma(\zeta),
\end{equation*}
where the function $\widehat{h}_{\varepsilon}$ is given by
\begin{equation}\label{operator inside integral for asymptotic expansion}
	\widehat{h}_{\varepsilon}(\zeta)=\left(\tfrac{1}{2}I+K_{\varepsilon}
			+\widetilde{D}_{\varepsilon}\right)^{-1}\left(S_{\varepsilon}g+\widetilde{S}_{\varepsilon}g\right)(z+\varepsilon \zeta)
\end{equation}
For $x,\zeta\in\mathbb{R}^{d}$ with $x\neq 0$ and $\varepsilon$ sufficiently small, there holds
\begin{equation}\label{taylor expansion gradient fundamental solution}
	\nabla_x \varGamma(x+\varepsilon\zeta)=\nabla_x \varGamma(x)+O(\varepsilon),
\end{equation}
therefore, taking advantage of the expansion \eqref{some expansion},
\begin{equation*}
	\begin{aligned}
	J_2&=\varepsilon^{d-1}\int_{\partial B}\frac{\partial}{\partial n_{\zeta}}\varGamma(x-z)
		\widehat{h}_{\varepsilon}(\zeta)\,d\sigma(\zeta)+O(\varepsilon^{d})\\
	&=\varepsilon^{d}\int_{\partial B}\frac{\partial}{\partial n_{\zeta}}\varGamma(x-z)
		\left(\tfrac{1}{2}I+K_B\right)^{-1}S_B\widehat{g}(\zeta)\,d\sigma(\zeta)+O(\varepsilon^{d+1}).
	\end{aligned}
\end{equation*}
Collecting the expansions for $J_1$ and $J_2$, we deduce \eqref{asymptotic expansion}.
\end{proof}

We show that the term $\left(\frac{1}{2}I+K_B\right)^{-1}S_Bg(x)$, for $x\in\partial B$, represents the trace
of the solution of the external domain related to the set $B$ and with Neumann boundary condition given by $g$.
To this aim, we consider the problem
\begin{equation}\label{differential problem for U}
\Delta U=0\quad \text{in}\ \mathbb{R}^d\setminus B,\quad
\frac{\partial U}{\partial n}=g\quad  \text{on}\ \partial B,\quad
U\longrightarrow 0\quad  \text{as}\ |x|\rightarrow +\infty,
\end{equation}
where the cavity $B$ is such that $0\in B$.

\begin{proposition}\label{Proposition for trace}
Let us define $f(x):=U(x)\big|_{x\in\partial B}$, then
\begin{equation*}
\left(\frac{1}{2}I+K_B\right)^{-1}S_Bg(x)=f(x).
\end{equation*}
\end{proposition}

\begin{proof}
As done in the proof of Theorem \ref{theorem representation formula}, that is, applying second Green's identity to the fundamental solution $\varGamma$ and $U$ in the domain $B_R(0)\setminus B$, with $R$ sufficiently large, it can be proven that the representation formula for $U$ is
\begin{equation}
U(x)=S_Bg(x)-D_Bf(x),\quad x\in\mathbb{R}^d\setminus B.
\end{equation}
Therefore, from single and double layer potentials properties
\begin{equation*}
f(x)=S_Bg(x)-\left(-\frac{1}{2}I+K_B\right)f(x),\quad x\in\partial B,
\end{equation*} 
hence
\begin{equation*}
f(x)=\left(\frac{1}{2}I+K_B\right)^{-1}S_Bg(x),\quad x\in\partial B,
\end{equation*}
that is the assertion.
\end{proof}

Now, we want to consider a specific case of the Neumann condition on the boundary of the cavity $C_\varepsilon$
so to get an explicit expression of the asymptotic expansion in terms of the polarization tensor and the fundamental solution.
\begin{corollary} 
Given $p\in\mathbb{R}^d$, let the boundary datum given by
\begin{equation*}
	g=-p\cdot n.
\end{equation*}
Then, the following expansion holds
\begin{equation}\label{simplified asymptotic expansion}
	u_{\varepsilon}(x)=2\,\varepsilon^d\nabla\varGamma(x-z)\cdot M p+O(\varepsilon^{d+1}),
	\qquad x\in\mathbb{R}^{d-1},
\end{equation}
where $M$ is the symmetric positive definite tensor given by 
\begin{equation}\label{pol tensor}
	M:=\int_{\partial B}n_\zeta \otimes \left(\zeta+\Psi(\zeta)\right)d\sigma(\zeta)
\end{equation}
and the auxiliary function $\Psi$ has components $\Psi_i$, $i=1,\dots,d$, solving
\begin{equation*}
	\Delta \Psi_i=0\quad\text{in}\;\mathbb{R}^d\setminus B,\qquad
	\frac{\partial \Psi_i}{\partial n}=-n_i\quad\text{on}\;\partial B,\qquad
	\Psi_i\longrightarrow 0\quad\text{as}\;|x|\rightarrow +\infty.
\end{equation*}
\end{corollary}

\begin{proof}
Let us set
\begin{equation*}
	\begin{aligned}
	J_1&:=\nabla\varGamma(x-z)\cdot\int_{\partial B} n_{\zeta}
		\left(\frac{1}{2}I+K_{B}\right)^{-1}S_{B}[-p\cdot n](\zeta)\,d\sigma(\zeta),\\
	J_2&:=\nabla\varGamma(x-z)\cdot\int_{\partial B}\zeta\,p\cdot n_{\zeta}\,d\sigma(\zeta).
	\end{aligned}
\end{equation*}
Then, expansion \eqref{asymptotic expansion} with $g=-p\cdot n$ gives
\begin{equation}\label{asymptotic expansion for pn}
	\begin{aligned}
	\tfrac{1}{2}u_{\varepsilon}(x)&=-\varepsilon^{d-1}\varGamma(x-z)\int_{\partial B}p\cdot n_{\zeta}\,d\sigma(\zeta)
		+J_1+J_2+O(\varepsilon^{d+1})\\
	&= J_1+J_2+O(\varepsilon^{d+1})
	\end{aligned}
\end{equation}  
since divergence Theorem guarantees that the first term in the expansion for $u_{\varepsilon}$ is null.

From the equation \eqref{differential problem for U}, with $g=-p\cdot n$, since the problem for $U$ is linear,
we can decompose $U$ as $U=\sum_i U_i$
where the functions $U_i$, for $i=1,\cdots,d$, solve 
\begin{equation*}
	\Delta U_i=0\quad \text{in}\ \mathbb{R}^d\setminus B,\quad
	\frac{\partial U_i}{\partial n}=-p_in_i\quad \text{on}\ \partial B,\quad
	U_i\longrightarrow 0\quad \text{as}\ |x|\rightarrow +\infty.
\end{equation*}
From the definition of the functions $\Psi_i$, we deduce $U=p\cdot \Psi$.
Using Proposition \ref{Proposition for trace}, the term $J_1$ can be rewritten as
\begin{equation*}
	J_1=\nabla\varGamma(x-z)\cdot\int_{\partial B}\left(\Psi(\zeta)\cdot p\right)n_\zeta\, d\sigma(\zeta)
		=\nabla\varGamma(x-z)\cdot \int_{\partial B}\left(n_\zeta \otimes \Psi(\zeta)\right) p\,d\sigma(\zeta).
\end{equation*} 
To deal with the term $J_2$, we preliminarly observe that 
\begin{equation*}
	\int_{\partial B}\left(\zeta\otimes n_\zeta\right) \,d\sigma(\zeta)=\int_{\partial B}\left(n_\zeta \otimes \zeta\right)\,d\sigma(\zeta).
\end{equation*}
Indeed, for $n_\zeta=(n_{\zeta,1},\dots n_{\zeta,d})$, for any $i,j\in\{1,\dots,d\}$, it follows
\begin{equation*}
	\begin{aligned}
	\int_{\partial B}\zeta_i\,n_{\zeta,j}\,d\sigma(\zeta)
	&=\int_{\partial B}n_\zeta \cdot  \zeta_i e_j\, d\sigma(\zeta)
		=\int_{B}\textrm{div}\left(\zeta_ie_j\right)d\zeta=\int_{B}e_j\cdot e_i\,d\zeta\\
	&=\int_{B}\nabla(\zeta_j)\cdot e_i\,d\zeta=\int_{B}\text{div}\left(\zeta_je_i \right)\,d\zeta
		=\int_{\partial B}n_\zeta \cdot\zeta_j e_i \,d\sigma(\zeta)\\
	&=\int_{\partial B}n_{\zeta,i}\,\zeta_j\,d\sigma(\zeta)
\end{aligned}
\end{equation*}
where $e_j$ is the $j$-th unit vector of $\mathbb{R}^d$.
Thus, 
we get
\begin{equation*}
	J_2=\nabla\varGamma(x-z)\cdot\int_{\partial B} (\zeta\otimes n_{\zeta})\,p \,d\sigma(\zeta)
		=\nabla\varGamma(x-z)\cdot \int_{\partial B}\left(n_\zeta \otimes \zeta\right)\,p\,d\sigma(\zeta).
\end{equation*}
Collecting the expressions for $J_1$ and $J_2$, we obtain formula \eqref{simplified asymptotic expansion}. 

Symmetry of the tensor $M$, defined in \eqref{pol tensor}, follows from
\begin{equation*}
	\begin{aligned}
	\int_{\partial B}\Psi_i(\zeta) n_{\zeta,j}\, d\sigma(\zeta)
		&=-\int_{\partial B}\Psi_i(\zeta) \frac{\partial \Psi_j}{\partial n}(\zeta)\, d\sigma(\zeta)\\
		&=\int_{\mathbb{R}^d\setminus B}\textrm{div}\left(\Psi_i(\zeta)\nabla \Psi_j(\zeta)\right)\, d\zeta
		=\int_{\mathbb{R}^d\setminus B}\nabla\Psi_i(\zeta)\cdot \nabla \Psi_j(\zeta)\, d\zeta
	\end{aligned}
\end{equation*}
where the last term is obviously symmetric.
Taking $\eta\in\mathbb{R}^d$, we consider
\begin{equation*}
\eta\cdot M\eta=\int_{\partial B}(n_{\varepsilon}\cdot\eta)(\zeta\cdot\eta)\, d\sigma(\zeta)+\int_{\partial B}(n_{\varepsilon}\cdot\eta)(\Psi(\zeta)\cdot \eta)\, d\sigma(\zeta)=I_1+I_2.
\end{equation*}
The positivity of the tensor follows from the divergence theorem and integration by parts, in fact
\begin{equation*}
I_1=\int_{B}\textrm{div}((\zeta\cdot\eta) \eta)\, d\zeta=\int_{B}\eta \cdot \nabla(\zeta\cdot\eta)\, d\zeta=\int_{B}|\eta|^2\, d\zeta=|\eta|^2|B|
\end{equation*}
In the same way, by the definition of the function $\Psi$
\begin{equation*}
\begin{aligned}
I_2=-\int_{\partial B}\frac{\partial}{\partial n}(\Psi(\zeta)\cdot\eta)(\Psi(\zeta)\cdot\eta)\, d\sigma(\eta)&=\int_{\mathbb{R}^d\setminus B}\textrm{div}\left((\Psi(\zeta)\cdot\eta)\nabla(\Psi(\zeta)\cdot\eta)\right)\, d\sigma(\zeta)\\
&=\int_{\mathbb{R}^d\setminus B}\big|\nabla(\Psi(\zeta)\cdot\eta)\big|^2\, d\zeta
\end{aligned}
\end{equation*}
The sum of $I_1$ and $I_2$ gives the positivity.
\end{proof}

For specific forms of the cavity $C$, the auxiliary function $\Psi$ can be determined explicitly,
providing a corresponding explicit formula for the polarization tensor $M$.
The basic case is the one of a spherical cavity (see \cite{Friedman-Vogelius}).
If $B=\{x\in\mathbb{R}^d\,:\,|x|<1\}$, then a direct calculation shows that, for $i=1,2,3$,
there holds $\Psi_i(x)=x_i/((d-1)|x|^n)$, and thus
\begin{equation*}
	\Psi_i(\zeta)=\frac{1}{d-1}\,\zeta_i,\qquad \zeta\in\partial B.
\end{equation*}
As a consequence, the polarization tensor is a multiple of the identity and, precisely,
\begin{equation*}
	M=\frac{3}{2}|B| I=2\pi I.
\end{equation*}
Then, the asymptotic expansion \eqref{simplified asymptotic expansion} becomes
\begin{equation*}
	u_{\varepsilon}(x)=4\pi\varepsilon^d\nabla\varGamma(x-z)\cdot p+O(\varepsilon^{d+1}),
	\qquad x\in\mathbb{R}^{d-1}.
\end{equation*}
Explicit formulas can be provided also in the case of ellipsoidal cavities
(see \cite{Ammari1, Ammari-Kang, Ammari-Kang1}).
	
In general, for given shapes of the cavity $C$, such auxiliary function can be numerically
approximated and, thus, the first term in the expansion \eqref{simplified asymptotic expansion}
can be considered as known in practical cases.

\subsection*{Acknowledgements}
C.Mascia has been partially supported by the italian Project FIRB 2012
``Dispersive dynamics: Fourier Analysis and Variational Methods''.

\end{document}